\documentclass{sig-alternate}

\usepackage{amsthm,amssymb,amsmath,stmaryrd}
\usepackage{graphics}
\usepackage{bbm}
\usepackage{tikz}
\usepackage[plainpages=false,pdfpagelabels=false,colorlinks=true,citecolor=blue,hypertexnames=false]{hyperref}
\usepackage{color}
\usepackage{graphicx}
\usepackage{amsmath}
\usepackage{amssymb}
\usepackage{enumerate}
\usepackage{mathrsfs}
\usepackage{makecell}
\usepackage[noend]{algpseudocode}
\usepackage{algorithmicx,algorithm}
\usepackage{multirow}
\usepackage{mathptmx}

\usepackage{dsfont}
\usepackage{epsfig}
\usepackage{mathrsfs}
\usepackage{amsmath}
\usepackage{cases}

\newtheorem{theorem}{Theorem}[section]
\newtheorem{example}[theorem]{Example}

\newtheorem{definition}[theorem]{Definition}
\newtheorem{remark}[theorem]{Remark}

\makeatletter \@addtoreset{equation}{section}

\newcommand{\R}{\mathbb{R}}
\newcommand{\N}{\mathbb{N}}

\newcommand{\x}{\mathbf{x}}

\newcommand{\br}{\mathbf{r}}

\newcommand{\be}{\mathbf{e}}
\newcommand{\bo}{\boldsymbol{\omega}}

\newcommand{\aaa}{\boldsymbol{\alpha}}

\newcommand{\bbb}{\boldsymbol{\beta}}

\newcommand{\A}{\mathscr{A}}
\newcommand{\B}{\mathscr{B}}

\newcommand{\supp}{\hbox{\rm{supp}}}

\newcommand{\New}{\hbox{\rm{New}}}
\newcommand{\Conv}{\hbox{\rm{conv}}}

\begin{document}
\clubpenalty=10000 \widowpenalty =10000

\title{A New Sparse SOS Decomposition Algorithm Based on Term Sparsity
\titlenote{This work was supported partly by NSFC under grants 61732001 and 61532019.}}

\numberofauthors{1}

\author{\medskip
Jie Wang, Haokun Li and Bican Xia\\
       \smallskip
       \affaddr{School of Mathematical Sciences, Peking University}\\
       \smallskip
      \email{wangjie212@pku.edu.cn, ker@protonmail.ch, xbc@math.pku.edu.cn}
}

\maketitle
\begin{abstract}
A new sparse SOS decomposition algorithm is proposed based on a new sparsity pattern, called cross sparsity patterns. The new sparsity pattern focuses on the sparsity of terms and thus is different from the well-known correlative sparsity pattern which focuses on the sparsity of variables though the sparse SOS decomposition algorithms based on these two sparsity patterns both take use of chordal extensions/chordal decompositions. Moreover, it is proved that the SOS decomposition obtained by the new sparsity pattern is always a refinement of the block-diagonalization obtained by the sign-symmetry method. 
Various experiments show that the new algorithm dramatically saves the computational cost compared to existing tools and can handle some really huge polynomials.
\end{abstract}





\section{Introduction}
Certificates of nonnegative polynomials and polynomial optimization problems (POPs) arise from many fields such as mathematics, control, engineering, probability, statistics and physics. A classical method for these problems is using sums of squares (SOS). For a polynomial $f\in\R[\x]=\R[x_1,\ldots,x_n]$ and a given monomial basis $M=\{\x^{\bo_1},\ldots,\x^{\bo_r}\}$, the SOS condition for $f$ can be converted to the problem of deciding if there exists a positive semidefinite matrix $Q$ (Gram matrix) such that $f=M^TQM$ which can be effectively solved by semidefinite programming (SDP) \cite{pa,pa1}.

When the given polynomial has many variables and a high degree, the corresponding SDP problem is hard to be dealt with by existing SDP solvers due to the very large size of the corresponding SDP matrix. On the other hand, most polynomials coming from practice have certain structures including symmetry and sparsity. So it is very important to take full advantage of structures of polynomials to reduce the size of corresponding SDP problems. In recent years, a lot of work has been done on this subject.

In the literature, there are three kinds of approaches to reduce computations by exploiting sparsity. One approach is reducing the size of the monomial basis $M$; such techniques include computing Newton polytopes \cite{re}, using the diagonal inconsistency \cite{lo1}, the iterative elimination method \cite{ko}, and the facial reduction \cite{pe}.
The second approach is exploiting the non-diagonal sparsity of the Gram matrix $Q$; such techniques include using the sparsity of variables \cite{ca,ma,nie,waki,we}, using the symmetry property \cite{ga}, using the split property \cite{dai}, and minimal coordinate projections \cite{pe1}. The third approach is exploiting the sparsity of constrained conditions of corresponding SDP problems, such as coefficient matching conditions \cite{be1,he1,yang1}.

In this paper, a new sparse SOS decomposition algorithm is proposed based on a new sparsity pattern, called {\em cross sparsity patterns}. 
Given a polynomial $f\in\R[\x]$ with the support $\A\subseteq\N^n$ and a monomial basis $M=\{\x^{\bo_1},\ldots,\x^{\bo_r}\}$, the cross sparsity pattern associated with $\A$ is represented by an $r\times r$ symmetric $(0,1)$-matrix $R_{\A}$ whose elements are defined by
\begin{equation}
R_{ij}=\begin{cases}
1, &\bo_i+\bo_j\in\mathscr{A}\cup2\B,\\
0, &\textrm{otherwise},
\end{cases}
\end{equation}
where $2\B=\{2\bo_1,\ldots,2\bo_r\}$.

It can be seen that the new sparsity pattern focuses on the sparsity of {\em terms} and thus is different from the well-known correlative sparsity pattern \cite{waki} which focuses on the sparsity of {\em variables}. For example, for a polynomial $f\in\R[\x]$, if $f$ contains a term involving all variables $x_1,\ldots,x_n$, then $f$ is not sparse in the sense of correlative sparsity patterns and hence the corresponding SDP matrix for the SOS decomposition of $f$ cannot be block-diagonalized. But $f$ may still be sparse in the sense of cross sparsity patterns (see Example \ref{ex}).

Following the chordal sparsity approaches, we associate the matrix $R_{\A}$ with an undirected graph $G(V_\mathscr{A},E_\mathscr{A})$ where $$V_\mathscr{A}=\{1,2,\ldots,r\}~ {\rm and}~ E_\mathscr{A}=\{\{i,j\}\mid i,j\in V_\mathscr{A}, i<j, R_{ij}=1\}$$ and generate a chordal extension of $G(V_\mathscr{A},E_\mathscr{A})$. Then as usual, we use matrix decompositions for positive semidefinite matrices with chordal sparsity patterns to construct sets of supports for a blocking SOS decomposition. We prove that the blocking SOS decomposition obtained by cross sparsity patterns is always a refinement of the block-diagonalization obtained by the sign-symmetry method \cite{lo1}.



We test the new algorithm on various examples. It turns out that the new algorithm dramatically reduces the computational cost compared to existing tools and can handle really huge polynomials which are unsolvable by any existing SDP solvers even exploiting sparsity.

The rest of this paper is organized as follows. Section 2 introduces some basic notions from nonnegative polynomials and graph theory. Section 3 defines a cross sparsity pattern matrix and a cross sparsity pattern graph associated with a sparse polynomial. We show that how we can exploit this sparsity pattern to obtain a block SOS decomposition for a sparse nonnegative polynomial. Moreover, we compare our approach with other methods to exploit sparsity in SOS decompositions, including correlative sparsity patterns and sign-symmetries. We discuss in Section 4 when the sparse SOS relaxation obtain the same optimal values as the dense SOS relaxation for polynomial optimization problems. The algorithm is given in Section 5. Section 6 includes numerical results on various examples. We show that the proposed {\tt SparseSOS} algorithm exhibits a significantly better performance in practice. Finally, the paper is concluded in Section 7.

\section{Preliminaries}
\subsection{Nonnegative polynomials}
Let $\R[\x]=\R[x_1,\ldots,x_n]$ be the ring of real $n$-variate polynomials. For a finite set $\A\subset\N^n$, we denote by $\Conv(\A)$ the convex hull of $\A$, and by $V(\A)$ the vertices of the convex hull of $\A$. Also we denote by $V(P)$ the vertex set of a polytope $P$. A polynomial $f\in\R[\x]$ can be written as $f(\x)=\sum_{\aaa\in\A}c_{\aaa}\x^{\aaa}$ with $c_{\aaa}\in\R, \x^{\aaa}=x_1^{\alpha_1}\cdots x_n^{\alpha_n}$. The support of $f$ is defined by $\supp(f)=\{\aaa\in\A\mid c_{\aaa}\ne0\}$, the degree of $f$ is defined
by $\deg(f)=\max\{\sum_{i=1}^n\alpha_i:\aaa\in\supp(f)\}$, and the Newton polytope of $f$ is defined as $\New(f)=\Conv(\{\aaa:\aaa\in\supp(f)\})$.

A polynomial $f\in\R[\x]$ which is nonnegative over $\R^n$ is called a {\em nonnegative polynomial}. The class of nonnegative polynomials is denoted by PSD, which is a convex cone.

A vector $\aaa\in\N^n$ is {\em even} if $\alpha_i$ is an even number for $i=1,\ldots,n$. A necessary condition for a polynomial $f(\x)$ to be nonnegative is that every vertex of its Newton polytope is an even vector, i.e. $V(\New(f))=V(\supp(f))\subseteq(2\N)^n$ \cite{re}.

For a nonempty finite set $\B\subseteq\N^n$, $\R[\B]$ denotes the set of polynomials in $\R[\x]$ whose supports are contained in $\B$, i.e., $\R[\B]=\{f\in\R[\x]\mid\supp(f)\subseteq\B\}$ and we use $\R[\mathscr{B}]^2$ to denote the set of polynomials which are sums of squares of polynomials in $\R[\B]$. The set of $r\times r$ symmetric matrices is denoted by $S^r$ and the set of $r\times r$ positive semidefinite matrices is denoted by $S_+^r$. Let $\x^{\mathscr{B}}$ be the $|\mathscr{B}|$-dimensional column vector consisting of elements $\x^{\bbb},\bbb\in\mathscr{B}$, then
\begin{equation*}
\R[\mathscr{B}]^2=\{(\x^{\mathscr{B}})^TQ\x^{\mathscr{B}}\mid Q\in S_+^{|\B|}\},
\end{equation*}
where the matrix $Q$ is called the {\em Gram matrix}.

\subsection{Chordal graphs}
We introduce some basic notions from graph theory. A {\em graph} $G(V,E)$ consists of a set of nodes $V=\{1,2,\ldots,r\}$ and a set of edges $E\subseteq V\times V$. A graph $G(V,E)$ is said to be {\em undirected} if and only if $(i,j)\in E\Leftrightarrow (j,i)\in E$. A {\em cycle} of length $k$ is a sequence of nodes $\{v_1,v_2,\ldots,v_k\}\subseteq V$ with $(v_k,v_1)\in E$ and $(v_i, v_{i+1})\in E$, for $i=1,\ldots,k-1$. A {\em chord} in a cycle $\{v_1,v_2,\ldots,v_k\}$ is an edge $(v_i, v_j)$ that joins two nonconsecutive nodes in the cycle.

An undirected graph is called a {\em chordal graph} if all its cycles of length at least four have a chord.
Chordal graphs include some common classes of graphs, such as complete graphs, line graphs and trees, and have applications in sparse matrix theory. Note that any non-chordal graph $G(V,E)$ can always be extended to a chordal graph $\widetilde{G}(V,\widetilde{E})$ by adding appropriate edges to $E$, which is called a {\em chordal extension} of $G(V,E)$. A {\em clique} $C\subseteq V$ is a subset of nodes where $(i,j)\in E$ for any $i,j\in C,i\ne j$. If a clique $C$ is not a subset of any other clique, then it is called a {\em maximal clique}. It is known that maximal cliques of a chordal graph can be enumerated efficiently in linear time in the number of vertices and edges of the graph. See for example \cite{bp,fg,go} for the details. 

Given an undirected graph $G(V,E)$, we define an extended set of edges $E^{\star}:=E\cup\{(i,i)\mid i\in V\}$ that includes all selfloops. Then, we define the space of symmetric sparse matrices as
\begin{equation}\label{sec2-eq1}
S^r(E,0):=\{X\in S^r\mid X_{ij}=X_{ji}=0\textrm{ if }(i,j)\notin E^{\star}\}
\end{equation}
and the cone of sparse PSD matrices as
\begin{equation}\label{sec2-eq2}
S^r_+(E,0):=\{X\in S^r(E,0)\mid X\succeq0\}.
\end{equation}

Given a maximal clique $C_k$ of $G(V,E)$, we define a matrix $P_{C_k}\in \R^{|C_k|\times r}$ as
\begin{equation}\label{sec2-eq3}
(P_{C_k})_{ij}=\begin{cases}
1, &C_k(i)=j,\\
0, &\textrm{otherwise}.
\end{cases}
\end{equation}
where $C_k(i)$ denotes the $i$-th node in $C_k$, sorted in the natural ordering. Note that $X_k=P_{C_k}XP_{C_k}^T\in S^{|C_k|}$ extracts a principal submatrix defined by the indices in the clique $C_k$, and $X=P_{C_k}^TX_kP_{C_k}$ inflates a $|C_k|\times|C_k|$ matrix into a sparse $r\times r$ matrix. Then, the following theorem characterizes the membership to the set $S^r_+(E,0)$ when the underlying graph $G(V,E)$ is chordal.
\begin{theorem}[\cite{ag}]\label{sec2-thm}
Let $G(V,E)$ be a chordal graph and $\{C_1,\ldots,C_t\}$ be all of the maximal cliques of $G(V,E)$. Then $X\in S_+^r(E,0)$ if and only if there exist $X_k\in S_+^{|C_k|}$ for $k=1,\ldots,t$ such that $X=\sum_{k=1}^tP_{C_k}^TX_kP_{C_k}$.
\end{theorem}

\section{Exploiting term sparsity in SOS decompositions}\label{sec3}
A convenient but incomplete algorithm for checking global nonnegativity of multivariate polynomials, as introduced by Parrilo in \cite{pa}, is the use of sums of squares as a suitable replacement for nonnegativity. Given a polynomial $f(\x)\in\R[\x]$, if there exist polynomials $f_1(\x),\ldots,f_m(\x)$ such that
\begin{equation}\label{sec3-eq3}
f(\x)=\sum_{i=1}^mf_i(\x)^2,
\end{equation}
then we say $f(\x)$ is a {\em sum of squares} (SOS). The existence of an SOS decomposition of a given polynomial gives a certificate for its global nonnegativity. For $d\in\N$, let $\N^n_d:=\{\aaa\in\N^n\mid\sum_{i=1}^n\alpha_i\le d\}$ and assume $f\in\R[\N^n_{2d}]$. The SOS condition (\ref{sec3-eq3}) can be converted to the problem of deciding if there exists a positive semidefinite matrix $Q$ such that
\begin{equation}\label{sec3-eq7}
f(\x)=(\x^{\N^n_{d}})^TQ\x^{\N^n_{d}},
\end{equation}
which is a semidefinite programming (SDP) problem.

We say that a polynomial $f\in\R[\N^n_{2d}]$ is {\em sparse} if the number of elements in its support $\A=\supp(f)$ is much less than the number of elements in $\N^n_{2d}$ that forms a support of fully dense polynomials in $\R[\N^n_{2d}]$. When $f(\x)$ is a sparse polynomial in $\R[\N^n_{2d}]$, the size of the SDP problem (\ref{sec3-eq7}) can be reduced by eliminating redundant elements from $\N^n_{d}$. In fact, $\N^n_{d}$ in problem (\ref{sec3-eq7}) can be replaced by \cite{re}
\begin{equation}\label{sec3-eq8}
\B=\Conv(\{\frac{\aaa}{2}\mid\aaa\in V(\A)\})\cap\N^n\subseteq\N^n_{d}.
\end{equation}
There are also other methods to reduce the size of $\B$ further, see for example \cite{ko,pe}.

\subsection{Cross sparsity pattern}
To exploit the term sparsity of polynomials in SOS decompositions, we introduce the notion of cross sparsity patterns, which, roughly speaking, is measured by the different kinds of cross products of monomials arising in the objective polynomial $f(\x)$.
\begin{definition}
Let $f(\x)\in\R[\x]$ with $\supp(f)=\mathscr{A}$. Assume that $\x^{\B}=\{\x^{\bo_1},\ldots,\x^{\bo_r}\}$ is a monomial basis. An $r\times r$ {\em cross sparsity pattern matrix} $\mathbf{R}_{\mathscr{A}}=(R_{ij})$ is defined by
\begin{equation}\label{sec3-eq1}
R_{ij}=\begin{cases}
1, &\bo_i+\bo_j\in\mathscr{A}\cup2\B,\\
0, &\textrm{otherwise},
\end{cases}
\end{equation}
where $2\B=\{2\bo_1,\ldots,2\bo_r\}$.

Given a cross sparsity pattern matrix $\mathbf{R}_\mathscr{A}=(R_{ij})$, the graph $G(V_\mathscr{A},E_\mathscr{A})$ where $$V_\mathscr{A}=\{1,2,\ldots,r\}~ {\rm and}~ E_\mathscr{A}=\{\{i,j\}\mid i,j\in V_\mathscr{A}, i<j, R_{ij}=1\}$$ is called the {\em cross sparsity pattern graph}.
\end{definition}

To apply Theorem \ref{sec2-thm}, we generate a chordal extension $\widetilde{G}(V_\mathscr{A},\widetilde{E}_\mathscr{A})$ of the cross sparsity pattern graph $G(V_\mathscr{A},E_\mathscr{A})$ and use the extended
cross sparsity pattern graph $\widetilde{G}(V_\mathscr{A},\widetilde{E}_\mathscr{A})$ instead of $G(V_\mathscr{A},E_\mathscr{A})$.

\begin{remark}
Given a graph $G(V_\mathscr{A},E_\mathscr{A})$, there may be many different chordal extensions and choosing anyone of them is valid for deriving the sparse SOS decompositions presented in this paper. For example, we can add edges to all of the connected components of $G(V_\mathscr{A},E_\mathscr{A})$ such that every connected component becomes a complete subgraph to obtain a chordal extension. The chordal extension with the least number of edges is called the {\em minimum chordal extension}. Finding the minimum chordal extension of a graph is an NP-hard problem in general. Finding a chordal extension of a graph is equivalent to calculating the symbolic sparse Cholesky factorization of its adjacency matrix. The resulted sparse matrix represents a chordal extension. The minimum chordal extension corresponds to the sparse Cholesky factorization with the minimum fill-ins. Fortunately, several heuristic algorithms, such as the minimum degree ordering, are known to efficiently produce a good approximation. For more information on symbolic Cholesky factorizations with the minimum degree ordering and minimum chordal extensions, see \cite{am,be,he}.
\end{remark}

\subsection{Sparse SOS relaxations}\label{sec-ssos}
Given $\A\subseteq\N^n$ with $V(\A)\subseteq(2\N)^n$, assume that $\B$ is the support set of a monomial basis. Let the set of SOS polynomials supported on $\A$ be
\begin{equation*}
\Sigma(\mathscr{A}):=\{f\in\R[\A]\mid\exists Q\in S_+^r\textrm{ s.t. }f=(\x^{\mathscr{B}})^TQ\x^{\mathscr{B}}\}.
\end{equation*}
Generally the Gram matrix $Q$ for a sparse SOS polynomial $f(\x)$ can be dense. Let $G(V_\mathscr{A},E_\mathscr{A})$ be the cross sparsity pattern graph and $\widetilde{G}(V_\mathscr{A},\widetilde{E}_\mathscr{A})$ a chordal extension. To maintain the sparsity of $f(\x)$ in the Gram matrix $Q$, we consider a subset of SOS polynomials
\begin{equation*}
\widetilde{\Sigma}(\mathscr{A}):=\{f\in\R[\A]\mid\exists Q\in S_+^r(\widetilde{E}_{\mathscr{A}},0)\textrm{ s.t. }f=(\x^{\mathscr{B}})^TQ\x^{\mathscr{B}}\}.
\end{equation*}

By virtue of Theorem \ref{sec2-thm}, the following theorem gives the blocking SOS decompositions for polynomials in $\widetilde{\Sigma}(\mathscr{A})$.
\begin{theorem}\label{sec3-thm1}
Given $\A\subseteq\N^n$ with $V(\A)\subseteq(2\N)^n$, assume that $\B=\{\bo_1,\ldots,\bo_r\}$ is the support set of a monomial basis and a chordal extension of the cross sparsity pattern graph is $\widetilde{G}(V_\mathscr{A},\widetilde{E}_\mathscr{A})$. Let $C_1, C_2, \ldots, C_t\subseteq V_\mathscr{A}$ denote the maximal cliques of $\widetilde{G}(V_\mathscr{A},\widetilde{E}_\mathscr{A})$ and $\mathscr{B}_k=\{\bo_i\in\B\mid i\in C_k\}, k=1,2,\ldots,t$. Then, $f(\x)\in\widetilde{\Sigma}(\mathscr{A})$ if and only if there exist $f_k(\x)\in\R[\mathscr{B}_k]^2$ for $k=1,\ldots,t$ such that
\begin{equation}\label{sec3-eq9}
f(\x)=\sum_{k=1}^tf_k(\x).
\end{equation}
\end{theorem}
\begin{proof}
By Theorem \ref{sec2-thm}, $Q\in S_+^r(\widetilde{E}_{\mathscr{A}},0)$ if and only if there exist $Q_k\in S_+^{|C_k|},k=1,\ldots,t$ such that $Q=\sum_{k=1}^tP_{C_k}^TQ_kP_{C_k}$. So $f(\x)\in\widetilde{\Sigma}(\mathscr{A})$ if and only if there exist $Q_k\in S_+^{|C_k|},k=1,\ldots,t$ such that
\begin{align*}
f(\x)&=(\x^{\mathscr{B}})^T(\sum_{k=1}^tP_{C_k}^TQ_kP_{C_k})\x^{\mathscr{B}}\\
&=\sum_{k=1}^t(P_{C_k}\x^{\mathscr{B}})^TQ_k(P_{C_k}\x^{\mathscr{B}})\\
&=\sum_{k=1}^t(\x^{\mathscr{B}_k})^TQ_k\x^{\mathscr{B}_k},
\end{align*}
which is equivalent to that there exist $f_k(\x)\in\R[\mathscr{B}_k]^2$ for $k=1,\ldots,t$ such that $f(\x)=\sum_{k=1}^tf_k(\x)$.
\end{proof}

\subsection{Comparison with correlative sparsity patterns}
The notion of correlative sparsity patterns was introduced by Waki et al. \cite{waki} to exploit variable sparsity of polynomials in SOS programming, which also takes use of chordal extensions/chordal decompositions. An interpretation of correlative sparsity patterns in terms of the sparsity of Gram matrices was recently given in \cite{yang2}. It should be emphasized that the angles of correlative sparsity patterns and cross sparsity patterns to exploit sparsity are different. Correlative sparsity patterns focus on the sparsity of variables, while cross sparsity patterns focus on the sparsity of terms. For example, for a polynomial $f\in\R[\x]$, if $f$ contains a term involving all variables $x_1,\ldots,x_n$, then $f$ is not sparse in the sense of correlative sparsity patterns and hence the corresponding SDP matrix for the SOS decomposition of $f$ cannot be block-diagonalized. But $f$ may still be sparse in the sense of cross sparsity patterns.
\begin{example}\label{ex}
Consider the polynomial $f=x^2y^2+x^2+y^2+1-xy$. A monomial basis for $f$ is $\{1,x,y,xy,x^2,y^2\}$. The correlative sparsity pattern graph of $f$ is a complete graph, and hence the corresponding Gram matrix of $f$ cannot be blocked. On the other hand, the cross sparsity pattern graph of $f$ has three maximal cliques, corresponding to $\{1,x^2,y^2\}$, $\{1,xy\}$ and $\{x,y\}$ respectively. Hence, the corresponding Gram matrix of $f$ can be blocked into one $3\times3$ submatrix and two $2\times2$ submatrices.
\end{example}

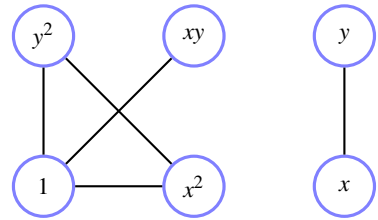
\begin{center}
\begin{tikzpicture}[every node/.style={circle, draw=blue!50, very thick, minimum size=8mm}]
\node (n1) at (0,0) {$1$};
\node (n2) at (2,0) {$x^2$};
\node (n3) at (4,0) {$x$};
\node (n4) at (0,2) {$y^2$};
\node (n5) at (2,2) {$xy$};
\node (n6) at (4,2) {$y$};
\draw[thick] (n1)--(n2);
\draw[thick] (n1)--(n4);
\draw[thick] (n2)--(n4);
\draw[thick] (n1)--(n5);
\draw[thick] (n3)--(n6);
\end{tikzpicture}
\end{center}

\subsection{Comparison with sign-symmetries}
In \cite{lo1}, sign-symmetries are exploited to block diagonalize sums of squares programming (\cite[Theorem 3]{lo1}), which is implemented in {\tt Yalmip}.
Given a polynomial $f\in\R[\x]$ with $\supp(f)=\mathscr{A}$. The {\em sign-symmetries} of $f$ are defined by all vectors $\br\in\{0,1\}^n$ such that $\br^T\aaa\equiv0$ $(\textrm{mod }2)$ for all $\aaa\in\A$.

By virtue of sign-symmetries, SOS programming can be blocked as follows.
\begin{theorem}[\cite{lo1}]\label{css}
Given a polynomial $f\in\R[\x]$ with $\supp(f)=\mathscr{A}$, assume that $\B=\{\bo_1,\ldots,\bo_r\}$ is the support set of a monomial basis and the sign-symmetries of $f$ are defined by the binary matrix $R=[\br_1,\ldots,\br_s]$. Then $\x^{\B}$ can be blocked in the SOS programming of $f$ and $\x^{\bo_i},\x^{\bo_j}$ belong to the same block if and only if $R^T\bo_i\equiv R^T\bo_j$ $(\textrm{mod }2)$.
\end{theorem}

We show in Theorem \ref{sec3-prop2} that the blocking decomposition obtained by cross sparsity patterns is always a refinement of the block-diagonalization obtained by sign-symmetries.
\begin{theorem}\label{sec3-prop2}
Given a polynomial $f\in\R[\x]$ with $\supp(f)=\mathscr{A}$, assume that $\B=\{\bo_1,\ldots,\bo_r\}$ is the support set of a monomial basis and the sign-symmetries of $f$ are defined by the binary matrix $R=[\br_1,\ldots,\br_s]$. Then the blocking decomposition obtained by cross sparsity patterns is a refinement of the block-diagonalization obtained by sign-symmetries in the SOS programming of $f$.
\end{theorem}
\begin{proof}
The block-diagonalization obtained by sign-symmetries can be represented by a graph $\overline{G}(V,\overline{E})$ with $V=\{1,\ldots,r\}$ and $(i,j)\in\overline{E}$ if and only if $R^T\bo_i\equiv R^T\bo_j$ $(\textrm{mod }2)$. Then by Theorem \ref{css}, the blocks obtained by sign-symmetries correspond to the connected components of $\overline{G}(V,\overline{E})$. To show that the blocking decomposition obtained by cross sparsity patterns is a refinement of the block-diagonalization obtained by sign-symmetries, we only need to prove that the cross sparsity pattern graph $G(V,E_\mathscr{A})$ is a subgraph of $\overline{G}(V,\overline{E})$, i.e. $E_\mathscr{A}\subseteq\overline{E}$.

By the definition of sign-symmetries, we have $R^T\aaa\equiv\mathbf{0}$ $(\textrm{mod }2)$ for all $\aaa\in\A$. By the definition of cross sparsity pattern graphs, $(i,j)\in E_\mathscr{A}$ if and only if $\bo_i+\bo_j\in\A\cup2\B$. So if $(i,j)\in E_\mathscr{A}$, then either $\bo_i+\bo_j\in\A$ or $\bo_i+\bo_j\in2\B$. In anyone of these two case, we always have $R^T(\bo_i+\bo_j)\equiv\mathbf{0}$ $(\textrm{mod }2)$, which is equivalent to $R^T\bo_i\equiv R^T\bo_j$ $(\textrm{mod }2)$. Thus $(i,j)\in\overline{E}$ as desired.
\end{proof}

\section{When do $\Sigma(\mathscr{A})$ and $\widetilde{\Sigma}(\mathscr{A})$ coincide}
Given $\A\subseteq\N^n$ with $V(\A)\subseteq(2\N)^n$, we define in Section \ref{sec-ssos} two sets of SOS polynomials: $\Sigma(\mathscr{A})$ and $\widetilde{\Sigma}(\mathscr{A})$. Generally we have $\Sigma(\mathscr{A})\supseteq\widetilde{\Sigma}(\mathscr{A})$. If $\Sigma(\mathscr{A})=\widetilde{\Sigma}(\mathscr{A})$, then the sparse SOS relaxation and the dense SOS relaxation obtain the same optimal value for the optimization of a polynomial $f$ with the support $\A$. The following theorem shows that in the quadratic case the equality $\Sigma(\mathscr{A})=\widetilde{\Sigma}(\mathscr{A})$ holds.
\begin{theorem}\label{sec3-prop1}
If for any $\aaa\in\A$, $\sum_{i=1}^n\alpha_i\le2$, then $\Sigma(\mathscr{A})=\widetilde{\Sigma}(\mathscr{A})$.
\end{theorem}
\begin{proof}
Suppose $f\in\Sigma(\mathscr{A})$ is a quadratic polynomial with $\supp(f)=\mathscr{A}$. Let $M=[1,x_1,\ldots,x_n]$ be a monomial basis and assume $f=M^TQM$ for a positive semidefinite matrix $Q=(q_{ij})_{i,j=0}^n$. Let $\mathbf{R}=(R_{ij})_{i,j=0}^n$ be the corresponding cross sparsity pattern matrix for $f$. To prove $\Sigma(\mathscr{A})\subseteq\widetilde{\Sigma}(\mathscr{A})$, we need to show $Q\in S_+^{n+1}(\widetilde{E}_{\mathscr{A}},0)$, or $Q\in S_+^{n+1}(E_{\mathscr{A}},0)$. Note that $Q\in S_+^{n+1}(E_{\mathscr{A}},0)$ is equivalent to the proposition that $R_{ij}=0$ implies $q_{ij}=0$ for all $i,j$. Let $\{\be_k\}_{k=1}^n$ be the standard basis. If $i=0,j>0$, from $R_{0j}=0$ we have $\be_j\notin\A$. If $i>0,j=0$, from $R_{i0}=0$ we have $\be_i\notin\A$. If $i,j>0,i\ne j$, from $R_{ij}=0$ we have $\be_i+\be_j\notin\A$. In anyone of these three cases, we have $q_{ij}=0$ as desired.
\end{proof}

\section{Algorithm}
According to Section \ref{sec3}, a sparse SOS decomposition procedure can be easily divided into the following four steps:
\begin{enumerate}
  \item Compute the support set of a monomial basis $\B$;
  \item Generate the cross sparsity pattern graph $G(V_\mathscr{A},E_\mathscr{A})$ and a chordal extension $\widetilde{G}(V_\mathscr{A},\widetilde{E}_\mathscr{A})$;
  \item Compute all of the maximal cliques of $\widetilde{G}(V_\mathscr{A},\widetilde{E}_\mathscr{A})$ and obtain the blocking SOS problem;
  \item Use an SDP solver to solve the blocking SOS problem.
\end{enumerate}

In step 1, we compute the support set of a monomial basis $\B$ following the method in \cite{lo1}.

In step 2, different chordal extensions will lead to different blocking SOS decompositions. When implementing this step, we obtain a chordal extension $\widetilde{G}(V_\mathscr{A},\widetilde{E}_\mathscr{A})$ by adding edges to $G(V_\mathscr{A},E_\mathscr{A})$ such that every connected component becomes a complete subgraph.

The above procedure is formally stated as Algorithm \ref{alg1} (named {\tt SparseSOS}) in the following. Obviously, since we use well-known and popular methods and tools for Step 1 and Step 4, the efficiency of {\tt SparseSOS} essentially depends on Step 2 and Step 3. That is, if we may decompose the original problems into smaller subproblems via Step 2 and Step 3, the computation cost will certainly be decreased because the SDP solver in Step 4 receives smaller inputs. We will show in the next section that {\tt SparseSOS} performs well on many examples.

\begin{algorithm}
\caption{{\tt SparseSOS}}\label{alg1}
\hspace*{0.02in}{\bf input}: a polynomial $f$ with $\supp(f)=\A$\\
\hspace*{0.02in}{\bf output}: a representation $f=\sum_{i=1}^mg_i^2$ or unknown
\begin{algorithmic}[1]
\State Compute the support set of a monomial basis $\B=\{\bo_1,\ldots,\bo_r\}$;
\State Generate the cross sparsity pattern graph $G(V_\mathscr{A},E_\mathscr{A})$;
\State Take the connected components $\{C_1,\ldots,C_t\}$ of $G(V_\mathscr{A},E_\mathscr{A})$ to obtain a chordal extension $\widetilde{G}(V_\mathscr{A},\widetilde{E}_\mathscr{A})$;
\State Solve the blocking SOS problem $$f=\sum_{k=1}^tf_k,~~~ f_k\in\R[\mathscr{B}_k]^2,~~~ (\ast)$$ where $\mathscr{B}_k=\{\bo_i\in\B\mid i\in C_k\}, k=1,2,\ldots,t$;
\State If ($\ast$) is feasible, then {\bf return} $f=\sum_{i=1}^mg_i^2$. Otherwise {\bf return} unknown.
\end{algorithmic}
\end{algorithm}

\section{Numerical Experiments}\label{sec-exa}
In this section, we give numerical results to illustrate the effectiveness of the algorithm {\tt SparseSOS}. The algorithm is implemented with C++ as a tool also named {\tt SparseSOS}. It turns out that {\tt SparseSOS} is extremely powerful and can deal with some really huge polynomials that cannot be handled by other tools.

\subsection{Versions and Commands}
Our tool {\tt SparseSOS} can be downloaded at

\centerline{https://gitlab.com/haokunli/sparsesos.}

All the examples in the following subsections can be downloaded there as well. We illustrate by a very simple example how to use {\tt SparseSOS}. Suppose we want to check whether the following polynomial is SOS by {\tt SparseSOS}:
\[\begin{array}{l}
36x_0^{10}x_1^2 + 4x_0^{10}x_2^2 + 81x_0^2x_1^8x_2^2 - 84x_0^8x_1^3 + 18x_0^2x_1^8x_2 + 49x_0^6x_1^4 \smallskip \\
+ x_0^2x_1^8 + 36x_0^4x_1^4x_2^2 + 4x_0^4x_1^4x_2 + 4x_0^6x_2^2.
\end{array}\]

First, express the polynomial by $+,-,*,\hat{ }$~, integers and variables in a file, say example.txt, as follows:
\begin{verbatim}
36*x0^10*x1^2 + 4*x0^10*x2^2 + 81*x0^2*x1^8*x2^2 -
84*x0^8*x1^3 + 18*x0^2*x1^8*x2 +49*x0^6*x1^4
+ x0^2*x1^8 + 36*x0^4*x1^4*x2^2 + 4*x0^4*x1^4*x2
+ 4*x0^6*x2^2.
\end{verbatim}
Then, we only need to type in:
\begin{verbatim}
is_sos example.txt
\end{verbatim}
to run {\tt SparseSOS} on the example.

{\tt SparseSOS} uses mosek 8.1 as an LP solver and csdp 6.2 as an SDP solver. In the following subsections, we compare the performance on some examples of {\tt SparseSOS} with that of {\tt Yalmip} \cite{lo}, {\tt SOSTOOLS} \cite{pap}, and {\tt SparsePOP} \cite{waki3} which also exploit sparsity in SOS decompositions. The versions of the tools and their LP and SDP solvers are listed here: {\tt Yalmip R20181012} (LP solver: gurobi 8.1; SDP solver: mosek 8.1), {\tt SOSTOOLS303} (SDP solver: sdpt 3.4) and {\tt SparsePOP301} (SDP solver: sdpt 3.4).

All numerical examples were computed on a 6-Core Intel Core i7-8750H@2.20GHz CPU with 16GB RAM memory and ARCH LINUX SYSTEM.

\begin{table*}[htbp]\label{table1}
  \begin{center}
  \caption{Notation}
   \begin{tabular}{|c|c|}
      \hline
      \#supp&the number of support monomials of a polynomial\\
      \hline
      \#block&the size of blocks obtained by SparseSOS\\
      \hline
      $i\times j$&$i$ blocks of size $j$\\
      \hline
      *&a failure information to obtain a SOS decomposition\\
      \hline
      OM&an out-of-memory error\\
      \hline
    \end{tabular}
  \end{center}
\end{table*}




\subsection{The polynomials $B_m$}\label{ex:bxm}
Let \[B_m=\left(\sum_{i=1}^{3m+2}x_i^2\right)\left(\left(\sum_{i=1}^{3m+2}x_i^2\right)^2-2\sum_{i=1}^{3m+2}x_i^2\sum_{j=1}^mx_{i+3j+1}^2\right),\] where we set $x_{3m+2+r}=x_r$. Note that $B_m$ is modified from \cite{pa}. For any $m\in\N\backslash\{0\}$, $B_m$ is homogeneous and is an SOS polynomial. For these $B_m$'s, {\tt SparseSOS} dramatically reduces the problem sizes and the computation time (see Table 2).

\begin{table*}[htbp]\label{table2}
  \begin{center}
    \caption{Results for $B_m$}
    \begin{tabular}{|c|c|c|c|c|c|c|c|c|c|}
      \hline
        & & \multicolumn{2}{c|}{\tt SparseSOS}& \multicolumn{2}{c|}{\tt Yalmip} & \multicolumn{2}{c|}{\tt SOSTOOLS}  & \multicolumn{2}{c|}{\tt SparsePOP} \\
       \hline
       $m$ & \#supp & \#block & time & \#block & time  & \#block & time  & \#block & time \\
      \hline
      1 & 35 & $5\times 5,10\times 1$ & 0.01s & $5\times 5,10\times 1$ & 0.45s
      & $1\times35$& 0.95s &$1\times56$ &0.54s \\
      \hline
      2 & 104 & $8\times 8,56\times 1$ & 0.04s & $8\times 8,56\times 1$ & 0.95s
      & $1\times120$ &2.59s & $1\times165$ &4.66s\\
      \hline
      3 & 242 & $11\times 11,165\times 1$ & 0.15s& $11\times 11,165\times 1$ & 1.18s &$1\times286$ & 34.00s& $1\times364$ & 93.9s\\
      \hline
      4 & 476 & $14\times 14,364\times 1$ & 0.45s & $14\times 14,364\times 1$ &  2.94s& $1\times560$ & 423s & $1\times680$ &764s \\
      \hline
      5 & 833 & $17\times 17,680\times 1$ & 1.56s & $1\times969$ & OM &$1\times969$ & OM &\multicolumn{2}{c|} {OM}\\
      \hline
      10& 5408 & $32\times 32, 4960\times 1$ & 65.55s &\multicolumn{6}{c|}{}\\
      \hline
    \end{tabular}
  \end{center}
\end{table*}
\begin{remark}
It is easy to see that, for $m\le 4$, {\tt SOSTOOLS} and {\tt SparsePOP} cannot block-diagonalize the corresponding Gram matrices for $B_m$ while {\tt Yalmip} and our tool {\tt SparseSOS} reduce the Gram matrices to smaller submatrices of the same size. That is the reason why {\tt Yalmip} and {\tt SparseSOS} cost much less time on those problems. For $m\geq 5$, only {\tt SparseSOS} can work out results and {\tt Yalmip} fails to obtain a block-diagonalization. 
\end{remark}

\subsection{MCP polynomials $P_{i,j}$}\label{ex:MCP}
Monotone Column Permanent (MCP) Conjecture was given in \cite{ha}. In the dimension $4$, this conjecture is equivalent to decide whether particular polynomials $p_{1,2}, p_{1,3}, p_{2,2}, p_{2,3}$ are nonnegative (the definitions of $p_{i,j}$ can be found in \cite{ka}). Actually, it was proved that every $p_{i,j}$ multiplied by a small particular polynomial is an SOS polynomial (\cite{ka}). Let
\begin{align*}
&P_{1,2}=(a^2 + 2 b^2 + c^2)\cdot p_{1,2},\\
&P_{1,3}=p_{1,3},\\
&P_{2,2}=(a^2 + 2 b^2 + c^2)\cdot p_{2,2},\\
&P_{2,3}=(a^2 + 2 b^2 + c^2)\cdot p_{2,3}.
\end{align*}
We use {\tt SparseSOS} to certify nonnegativity of $P_{1,2}, P_{1,3}, P_{2,2}, P_{2,3}$. The result is listed in Table 3.

\begin{table*}[htbp]
  \caption{Results for $P_{i,j}$}
  \begin{center}
    \begin{tabular}{|c|c|c|c|c|c|c|c|c|c|}
      \hline
      & & \multicolumn{2}{c|}{\tt SparseSOS}& \multicolumn{2}{c|}{\tt Yalmip} & \multicolumn{2}{c|}{\tt SOSTOOLS}  & \multicolumn{2}{c|}{\tt SparsePOP} \\
       \hline
        & \#supp & \#block & time & \#block & time  & \#block & time  & \#block & time \\
      \hline
      \multirow{2}*{$P_{1,2}$} &\multirow{2}*{159} & $1\times 15,2\times 12,7\times 4,$ & \multirow{2}*{0.29s} & $1\times 15,2\times 12,7\times 4,$ &\multirow{2}*{1.86s}&\multirow{2}*{$1\times 77$}&\multirow{2}*{2.39s}&\multirow{2}*{$1\times 112$}&\multirow{2}*{2.56s}\\

      ~&~ &$1\times 3,2\times 2,3\times 1$&~&$1\times 3,2\times 2,3\times 1$& ~& ~&~& ~&~\\
      \hline
      \multirow{2}*{$P_{1,3}$} &\multirow{2}*{53} & $1\times 8,4\times 3,$ & \multirow{2}*{0.08s} & $1\times 8,4\times 3,$ &\multirow{2}*{0.41s}&\multirow{2}*{$1\times 29$}&\multirow{2}*{0.86s}&\multirow{2}*{$2\times 30,1\times 29$}&\multirow{2}*{0.52s}\\

      ~&~ &$2\times 2,5\times 1$&~&$2\times 2,5\times 1$& ~& ~&~& ~&~\\

       \hline
      \multirow{2}*{$P_{2,2}$} &\multirow{2}*{144} & $3\times 12,2\times 4,$ & \multirow{2}*{0.27s} & $3\times 12,2\times 4,$ &\multirow{2}*{0.40s}&\multirow{2}*{$1\times 25$}&\multirow{2}*{*}&\multirow{2}*{$1\times 97$}&\multirow{2}*{2.23s}\\

      ~&~ &$8\times 2,2\times 1$&~&$8\times 2,2\times 1$& ~& ~&~& ~&~\\

       \hline
      \multirow{2}*{$P_{2,3}$} &\multirow{2}*{107} & $2\times 10,1\times 8,1\times 4,$ & \multirow{2}*{0.19s} & $2\times 10,1\times 8,1\times 4,$ &\multirow{2}*{0.40s}&\multirow{2}*{$1\times 53$}&\multirow{2}*{1.62s}&\multirow{2}*{$1\times 65,1\times 60$}&\multirow{2}*{1.48s}\\

      ~&~ &$1\times 3,8\times 2,2\times 1$&~&$1\times 3,8\times 2,2\times 1$& ~& ~&~& ~&~\\

       \hline

    \end{tabular}
  \end{center}
\end{table*}

\begin{remark}
When we use the 'sparse' option, {\tt SOSTOOLS} seems to make a mistake in computing a monomial basis for $P_{2,2}$ and fails to obtain a SOS decomposition for $P_{2,2}$.
\end{remark}

\subsection{Randomly generated polynomials}
Now we present the numerical results for randomly generated polynomials. A sparse randomly generated polynomial $$f=\sum_{i=1}^kf_i^2\in\textbf{randpoly}(n,d,k,p)$$ is constructed as follows: first generate a set of monomials $M$ in the set $\x^{\N^n_{d}}$ with probability $p$, and then randomly assign the elements of $M$ to $f_1,\ldots,f_k$ with random coefficients between $-10$ and $10$. We generate $18$ random polynomials $F_1,\ldots,F_{18}$ from $6$ different classes, where $$F_1,F_2,F_3\in\textbf{randpoly}(10,6,10,0.01),$$ $$F_4,F_5,F_6\in\textbf{randpoly}(10,6,10,0.015),$$ $$F_7,F_8,F_9\in\textbf{randpoly}(10,10,10,0.001),$$
$$F_{10},F_{11},F_{12}\in\textbf{randpoly}(10,8,20,0.002),$$
$$F_{13},F_{14},F_{15}\in\textbf{randpoly}(10,8,20,0.005)$$
and $$F_{16},F_{17},F_{18}\in\textbf{randpoly}(10,8,20,0.01).$$ See Table 4 for the performance of {\tt Yalmip} and {\tt SparseSOS} on these polynomials. Since {\tt SOSTOOLS} and {\tt SparsePOP} can hardly handle these polynomials, we do not list the performance of them in the table.
\begin{table*}[htbp]
  \label{tb:randpoly}
  \caption{The result for randomly generated polynomials}
  \begin{center}
  \scalebox{0.9}{
    \begin{tabular}{|c|c|c|c|c|c||c|c|c|c|c|c|}
      \hline
      & & \multicolumn{2}{c|}{\tt SparseSOS}& \multicolumn{2}{c||}{\tt Yalmip}&& & \multicolumn{2}{c|}{\tt SparseSOS}& \multicolumn{2}{c|}{\tt Yalmip} \\
       \hline
       & \#supp & \#block & time & \#block & time&& \#supp & \#block & time & \#block & time  \\
      \hline
      \multirow{2}*{$F_1$} &\multirow{2}*{590} & $187,5,$ & \multirow{2}*{179.2s} & \multirow{2}*{$  248$} &\multirow{2}*{315.60s}&\multirow{2}*{$F_4$} &\multirow{2}*{873} & $  303,  8,$ & \multirow{2}*{1850.54s} & \multirow{2}*{$  357$} &\multirow{2}*{OM}\\

      ~&~ &$6\times 2,44\times 1$&~&~& ~& ~&~ &$3\times 2,40\times 1$&~&~& ~\\
      \hline
      \multirow{2}*{$F_2$} &\multirow{2}*{310} & $  83,  3,$ & \multirow{2}*{4.42s} & \multirow{2}*{$131$} &\multirow{2}*{16.34s}&\multirow{2}*{$F_5$} &\multirow{2}*{709} & $  238,  4,$ & \multirow{2}*{633.51s} & \multirow{2}*{$  331$} &\multirow{2}*{OM}\\

      ~&~ &$4\times 2,37\times 1$&~&~& ~&~&~ &$4\times 3,1  2,55\times 1$&~&~& ~\\
      \hline
\multirow{2}*{$F_3$} &\multirow{2}*{504} & $  162,  6, 4,$ & \multirow{2}*{63.86s} & \multirow{2}*{$  218$} &\multirow{2}*{116.09s}&\multirow{2}*{$F_6$} &\multirow{2}*{927} & $  231,  3,$ & \multirow{2}*{470.40s} & \multirow{2}*{$  261$} &\multirow{2}*{297.40s}\\

      ~&~ &$6\times 2,34\times 1$&~&~& ~&~&~ &$2\times 2,23\times 1$&~&~& ~\\

       \hline
       \hline
      \multirow{2}*{$F_7$} &\multirow{2}*{1344} & $  4658,  7,2\times 5,3\times 4$ & \multirow{2}*{OM} & \multirow{2}*{$  4769$} &\multirow{2}*{OM}&
      \multirow{2}*{$F_{10}$} &\multirow{2}*{306} & $  110,  10,  6,3\times 4,$ & \multirow{2}*{29.95s} & \multirow{2}*{$  389$} &\multirow{2}*{OM}\\

      ~&~ &$7\times 3,16\times 2,29\times 1$&~&~& ~&
      ~&~ &$5\times 3,22\times 2,192\times 1$&~&~& ~\\

       \hline
      \multirow{2}*{$F_8$} &\multirow{2}*{1392} & $  5012,  5,3,$ & \multirow{2}*{OM} & \multirow{2}*{$  5046$} &\multirow{2}*{OM}&
      \multirow{2}*{$F_{11}$} &\multirow{2}*{255} & $  62,  8,  5, 4,$ & \multirow{2}*{32.09s} & \multirow{2}*{$  220$} &\multirow{2}*{185.35s}\\

      ~&~ &$3\times 2,20\times 1$&~&~& ~&
      ~&~ &$2\times 3,2\times 2,131\times 1$&~&~& ~
      \\

       \hline
           \multirow{2}*{$F_9$} &\multirow{2}*{1845} & $  4528,  7,3,$ & \multirow{2}*{OM} & \multirow{2}*{$  4576$} &\multirow{2}*{OM}&
           \multirow{2}*{$F_{12}$} &\multirow{2}*{228} & $  56,  13,2\times 6,2\times 4,$ & \multirow{2}*{11.24s} & \multirow{2}*{$  232$} &\multirow{2}*{200.43s}\\


      ~&~ &$5\times 2,28\times 1$&~&~& ~&
       ~&~ &$4\times 3,12\times 2,107\times 1$&~&~& ~\\

       \hline
              \hline
      \multirow{2}*{$F_{13}$} &\multirow{2}*{1446} & $  2394,  3,$ & \multirow{2}*{OM} & \multirow{2}*{$  2450$} &\multirow{2}*{OM}&
      \multirow{2}*{$F_{16}$} &\multirow{2}*{4777} & \multirow{2}*{$  8866$}  & \multirow{2}*{OM} & \multirow{2}*{$  8866$} &\multirow{2}*{OM}\\

      ~&~ &$8\times 2,37\times 1$&~&~& ~&
      ~&~ &~&~&~& ~\\

       \hline
      \multirow{2}*{$F_{14}$} &\multirow{2}*{1636} & $  2154,  3,$ & \multirow{2}*{OM} & \multirow{2}*{$  2206$} &\multirow{2}*{OM}&
      \multirow{2}*{$F_{17}$} &\multirow{2}*{4959} &\multirow{2}*{$  8415$}  & \multirow{2}*{OM} & \multirow{2}*{$  8415$} &\multirow{2}*{OM}\\

      ~&~ &$4\times 2,43\times 1$&~&~& ~&
      ~&~ &~&~&~& ~
      \\

       \hline
           \multirow{2}*{$F_{15}$} &\multirow{2}*{1085} & $  1800,  8,  4, 6\times 3,$ & \multirow{2}*{OM} & \multirow{2}*{$  1980$} &\multirow{2}*{OM}&
           \multirow{2}*{$F_{18}$} &\multirow{2}*{4869} & \multirow{2}*{$  8712$}& \multirow{2}*{OM} & \multirow{2}*{$  8712$} &\multirow{2}*{OM}\\


      ~&~ &$23\times 2,104\times 1$&~&~& ~&
       ~&~ &~&~&~& ~\\

       \hline

    \end{tabular}}\\
  {\small In this table, $1\times j$ is denoted by $j$ for short. For example, the \#block data $248$ of {\tt Yalmip} for $F_1$ stands for one block of size $248$.}
  \end{center}
\end{table*}

\begin{remark}
From Table 4, we can see that {\tt SparseSOS} obtains block-diagonalizations for $F_1,...,F_{15}$ while {\tt Yalmip} fails for all these polynomials. For polynomials $F_{16},F_{17},F_{18}$, both {\tt SparseSOS} and {\tt Yalmip} cannot obtain block-diagonalizations.

For $F_1,...,F_6$ and $F_{10},...,F_{12}$, {\tt SparseSOS} succeeds in obtaining the final SOS decompositions while {\tt Yalmip} fails on $F_4,F_5$ and $F_{10}$. Furthermore, {\tt SparseSOS} is faster than {\tt Yalmip} on all these polynomials except $F_6$. We observe that the reason why {\tt Yalmip} is faster on $F_6$ lies in the efficiency of SDP solvers. Mosek is faster on $F_6$ than csdp.

Although we select only three polynomials from each class of random polynomials, we notice that {\tt SparseSOS} performs similarly on polynomials from the same class. For example, for the classes $\textbf{randpoly}(10,6,10,0.01),$ $\textbf{randpoly}(10,6,10,0.015),$ and $\textbf{randpoly}\allowbreak(10,8,20,0.002),$ {\tt SparseSOS} succeeds in obtaining the final SOS decompositions. For the classes $\textbf{randpoly}(10,10,10,0.001)$ and $\textbf{randpoly}\allowbreak(10,8,20,0.005)$, {\tt SparseSOS} can obtain block-diagonali\-zations of the corresponding Gram matrices but cannot work out the final result. For the class $\textbf{randpoly}(10,8,20,0.01)$, {\tt SparseSOS} cannot obtain block-diagonalizations.
\end{remark}


%
%

\section{Conclusions}
We exploit the term sparsity of polynomials in SOS Programming by virtue of cross sparsity patterns and prove a sparse SOS decomposition theorem for sparse polynomials via PSD matrix decompositions with chordal sparsity patterns. Based on this, a new sparse SOS algorithm is proposed and is tested on various examples. The experimental results show that the new algorithm is efficient and extremely powerful. The algorithm can be combined with other simplification methods, e.g. \cite{be1}, to reduce computational costs further. We will apply the {\tt SparseSOS} algorithm to solve large scale unconstrained and constrained polynomial optimization problems in future work.

\bibliographystyle{plain}


\def\cprime{$'$}


\end{document}